\documentclass[11pt,twoside]{article}
\usepackage{amsmath,amsbsy,amsfonts}
\usepackage{graphicx,color}
\usepackage{upgreek}
\everymath{\displaystyle}
\newtheorem{theorem}{Theorem}[section]
\newtheorem{lemma}[theorem]{Lemma}
\newtheorem{proposition}[theorem]{Proposition}
\newtheorem{corollary}[theorem]{Corollary}

\newtheorem{claim}[theorem]{Claim}

\newenvironment{proof}[1][Proof]{\noindent\textbf{#1.} }{\ \rule{0.5em}{0.5em}}

\newcommand{\R}{\mathbb{R}}
\newcommand{\Z}{\mathbb{Z}}
\newcommand{\N}{\mathbb{N}}
\newcommand{\M}{\mathcal{M}}

\begin{document}

\title{\bf Existence and concentration phenomena for a class of indefinite variational problems with critical growth}

\author{Claudianor O. Alves\thanks{C. O. Alves was partially supported by CNPq/Brazil 301807/2013-2 and INCT-MAT, coalves@mat.ufcg.edu.br}\, , \, \ Geilson F. Germano \thanks{G. F. Germano was partially supported by CAPES, geilsongermano@hotmail.com}\vspace{2mm}
	\and {\small  Universidade Federal de Campina Grande} \\ {\small Unidade Acad\^emica de Matem\'{a}tica} \\ {\small CEP: 58429-900, Campina Grande - Pb, Brazil}\\}

\date{}
\maketitle

\begin{abstract} 
In this paper we are interested to prove the existence and concentration of ground state solution for the following class of problems
$$
-\Delta u+V(x)u=A(\epsilon x)f(u), \quad x \in \R^{N}, 
\eqno{(P)_{\epsilon}}
$$
where $N \geq 2$, $\epsilon>0$, $A:\R^{N}\rightarrow\R$ is a continuous function that satisfies
$$
0<\inf_{x\in\R^{N}}A(x)\leq\lim_{|x|\rightarrow+\infty}A(x)<\sup_{x\in\R^{N}}A(x)=A(0),\eqno{(A)}
$$
$f:\R\rightarrow\R$ is a continuous function having critical growth, $V:\R^{N}\rightarrow\R$ is a continuous and $\Z^{N}$--periodic function with $0\notin\sigma(\Delta+V)$.  By using variational methods, we prove the existence of solution for $\epsilon$ small enough. After that, we show that the maximum points of the solutions concentrate around of a maximum point of $A$.

	\vspace{0.3cm}
	
	
	\vspace{0.3cm}
	
\end{abstract}

\section{Introduction}

This paper concerns with the existence and concentration of ground state solution for the semilinear Schr\"odinger equation
$$
\left\{\begin{array}{l}
-\Delta u+V(x)u=A(\epsilon x)f(u), \quad x \in \R^{N}, \\
u\in H^{1}(\R^{N}),
\end{array}\right. \eqno{(P)_{\epsilon}}
$$
where $N \geq 2$, $\epsilon$ is a positive parameter, $f: \mathbb{R} \to \mathbb{R}$ is a continuous function with critical growth and $V,A: \mathbb{R} \to \mathbb{R}$ are continuous functions verifying some technical conditions. 

In whole this paper, $V$ is $\mathbb{Z}^N$-periodic with   
$$
0 \not\in \sigma(-\Delta + V), \quad \mbox{the spectrum of } \quad -\Delta +V,  \eqno{(V)}
$$
which becomes the problem strongly indefinite. Related to the function $A$, we assume that it is a continuous function satisfying 
$$
0< A_0=\inf_{x \in \R^{N}}A(x)\leq \displaystyle\lim_{|x|\rightarrow+\infty}A(x)=A_\infty <\sup_{x \in \R^{N}}A(x). \eqno{(A)}
$$		

The present article has as first motivation some recent articles that have studied the existence of ground state solution for related problems with $(P)_{\epsilon}$, more precisely for strongly indefinite problems of the type  
$$
\left\{\begin{array}{l}
-\Delta u+V(x)u=f(x,u), \quad x \in  \R^{N}, \\
u\in H^{1}(\R^{N}).
\end{array}\right. \eqno{(P_1)}
$$
In \cite{Szulkin1}, Kryszewski and Szulkin  have studied the existence of ground state solution for $(P_1)$ by supposing the condition $(V)$. Related to the function $f:\mathbb{R}^N \times \mathbb{R} \to \mathbb{R}$, they assumed that $f$ is continuous, $\mathbb{Z}^N$-periodic in $x$ with
$$
|f(x,t)| \leq c(|t|^{q-1}+|t|^{p-1}), \quad \forall t \in \mathbb{R} \quad \mbox{and} \quad x \in \mathbb{R}^N \eqno{(h_1)}
$$ 
and
$$
0<\alpha F(x,t) \leq tf(x,t) \quad \forall (x,t) \in \mathbb{R}^N \times \mathbb{R}^* , \quad F(x,t)=\int_{0}^{t}f(x,s)\,ds \eqno{(h_2)}
$$
for some $c>0$, $\alpha >2$ and $2<q<p<2^{*}$ where $2^{*}=\frac{2N}{N-2}$ if $N \geq 3$ and $2^{*}=+\infty$ if $N=1,2$. The above hypotheses guarantee that the energy functional associated with $(P_1)$  given by 
$$
J(u)=\frac{1}{2}\int_{\mathbb{R}^N}(|\nabla u|^{2}+V(x)|u|^{2})\,dx-\int_{\mathbb{R}^N}F(x,u)\,dx, \,\, \forall u \in H^{1}(\R^N),
$$
is well defined and belongs to $C^{1}(H^{1}(\mathbb{R}^N), \mathbb{R})$. By $(V)$, there is an equivalent inner product $\langle \;\; , \;\; \rangle $  in $H^{1}(\mathbb{R}^N)$ such that 
$$
J(u)=\frac{1}{2}\|u^+\|^{2}-\frac{1}{2}\|u^-\|^{2} -\int_{\mathbb{R}^N}F(x,u)\,dx,
$$
where $\|u\|=\sqrt{\langle u,u \rangle}$ and  $H^{1}(\mathbb{R}^N) = E^{+} \oplus E^-$ corresponds to the spectral decomposition of $- \Delta + V $ with respect to the
positive and negative part of the spectrum with $u = u^{+}+u^{-}$, where $u^{+} \in E^{+}$ and $u^{-} \in E^{-}$.  In order to show the existence of solution for $(P_1)$, Kryszewski and Szulkin  introduced a new and interesting generalized link theorem. In \cite{LiSzulkin},  Li and Szulkin have improved this generalized link theorem to prove the existence of solution for a class of strongly indefinite problem with $f$ being asymptotically linear at infinity.   

The link theorems above mentioned have been used in a lot of papers, we would like to cite   Chabrowski and Szulkin \cite{CS}, do \'O and Ruf \cite{DORUF},   Furtado and Marchi \cite{furtado}, Tang \cite{tang, tang2} and their references.

Pankov and Pfl\"uger \cite{Pankov-Pfluger} also have considered the existence of solution for problem $(P_1)$ with the same conditions considered in \cite{Szulkin1}, however the approach is based on an approximation technique of periodic function together with the linking theorem due to Rabinowitz \cite{Rabinowitz}. After, Pankov \cite{Pankov} has studied the  existence of solution for problems of the type  
$$
\left\{\begin{array}{l}
-\Delta u+V(x)u=\pm f(x,u), \quad x \in \R^{N}, \\
u\in H^{1}(\R^{N}),
\end{array}\right. \eqno{(P_2)}
$$
by supposing $(V)$, $(h_1)-(h_2)$ and employing the same approach explored in \cite{Pankov-Pfluger}.  In   \cite{Pankov} and \cite{Pankov-Pfluger}, the existence of ground state solution has been established by supposing that $f$ is $C^{1}$ and there is $\theta \in (0,1)$ such that
$$
0<t^{-1}f(x,t)\leq \theta f'_t(x,t), \quad \forall t \not=0 \quad \mbox{and} \quad x \in \mathbb{R}^N. \eqno({h_3})
$$
However, in  \cite{Pankov}, Pankov has found a ground state solution by minimizing the energy functional  $J$ on the set
$$
\mathcal{O}=\left\{u\in H^{1}(\R^{N})\setminus E^{-}\ ;\ J'(u)u=0\text{ and }J'(u)v=0,\forall\ v\in E^{-}\right\}. 
$$ 
The reader is invited to see that if $J$ is strongly definite, that is, when $E^{-}=\{0\}$,  the set $\mathcal{O}$ is exactly the Nehari manifold associated with $J$. Hereafter, we say that  $u_0 \in H^{1}(\mathbb{R}^{N})$ is a {\it ground state solution} if 
$$
J'(u_0)=0, \quad u_0 \in \mathcal{O} \quad \mbox{and} \quad J(u_0)=\inf_{w \in \mathcal{O}}J(w).
$$

In \cite{SW}, Szulkin and Weth have established the existence of ground state solution for problem $(P_1)$ by completing the study made in  \cite{Pankov}, in the sense that, they also minimize the energy functional on  $\mathcal{O}$, however they have used more weaker conditions on $f$, for example $f$ is continuous, $\mathbb{Z}^N$-periodic in $x$ and satisfies 
$$
|f(x,t)|\leq C(1+|t|^{p-1}), \;\; \forall t \in \mathbb{R} \quad \mbox{and} \quad x \in \mathbb{R}^N  \eqno{(h_4)}
$$ 
for some $C>0$ and $p \in (2,2^{*})$.
$$
f(x,t)=o(t) \,\,\, \mbox{uniformly in } \,\, x \,\, \mbox{as} \,\, |t| \to 0. \eqno{(h_5)}
$$
$$
F(x,t)/|t|^{2} \to +\infty \,\,\, \mbox{uniformly in } \,\, x \,\, \mbox{as} \,\, |t| \to +\infty, \eqno{(h_6)}
$$
and
$$
t\mapsto f(x,t)/|t| \,\,\, \mbox{is strictly increasing on} \,\,\, \mathbb{R} \setminus \{0\}. \eqno{(h_7)}
$$
The same approach has been used by Zhang, Xu and Zhang \cite{ZXZ, ZXZ2} to study a class of indefinite and asymptotically periodic problem.   

In \cite{AG}, Alves and Germano have studied the existence of ground state solution for problem $(P_1)$ by supposing the $f$ has a critical growth for $N \geq 2,$ while in \cite{AG2} the authors have established the existence and concentration of solution for problem $(P)_\epsilon$ by supposing that $f$ has a subcritical growth and $V,A$ verify the conditions  $(V)$ and $(A)$ respectively.

Motivated by results found \cite{AG,AG2}, in the present paper we intend to study the existence and concentration of solution for problem $(P)_\epsilon$ for the case where function $f$ has a critical growth. Since the critical growth brings a lost of compactness, we have established new estimates for the problem. Here, the concentration phenomena is very subtle, because we need to be careful to prove some estimates involving the $L^{\infty}$ norm of the solutions for $\epsilon$ small enough, for more details see Section 2.2 for $N \geq 3$,  and Section 3.3 for $N=2$. In additional to conditions $(V)$ and $(A)$ on the functions $V$ and $A$ respectively, we are supposing the following conditions on $f$: 

\vspace{0.5 cm}

\noindent\textbf{The Case $N\geq3$:}\\

In this case $f:\R\rightarrow\R$ is of the form 
$$
f(t)=\xi |t|^{q-1}t+|t|^{2^{*}-2}t, \quad \forall t \in \R; \leqno{(f_0)}
$$
with $\xi>0, q \in (2,2^*)$ and $2^*={2N}/{N-2}$.

\vspace{0.5 cm}

\noindent\textbf{The Case $N=2$:}\\

In this case $f:\R\rightarrow\R$ is a continuous function that satisfies
\begin{itemize}
	\item[$(f_1)$] $\frac{f(t)}{t}\rightarrow0$ as $t\rightarrow0$ ;
	\item[$(f_2)$] The function $t\mapsto\frac{f(t)}{t}$ is increasing on $(0,+\infty)$ and decreasing on $(-\infty,0)$;
	\item[$(f_3)$] There exists $\theta>2$ such that
	$$
	0<\theta F(t)\leq f(t)t, \quad \forall t \in \R \setminus \{0\}
	$$
	where
	$$
	F(t):=\int_{0}^{t}f(s)ds;
	$$
	\item[$(f_4)$] There exists $\Gamma>0$ such that $|f(t)|\leq\Gamma e^{4\pi t^{2}}$ for all $t \in \R$;
	\item[$(f_5)$] There exist $\tau>0$ and $q>2$  such that $F(t)\geq\tau|t|^{q}$ for all $t \in \R$.
\end{itemize}
The condition $(f_4)$ says that $f$ can have an  exponential critical growth. Here, we recall that a function $f$ has an exponential critical growth, if there is $\alpha_0>0$ such that   
$$
\lim_{|t| \to +\infty}\frac{|f(t)|}{e^{\alpha|t|^{2}}}=0, \;\; \forall \alpha > \alpha_0, 
\lim_{|t| \to +\infty}\frac{|f(t)|}{e^{\alpha|t|^{2}}}=+\infty, \;\; \forall \alpha < \alpha_0.
$$

\vspace{0.5 cm}

Our main theorem is the following 
\begin{theorem}  \label{T1}
	Assume $(V), (A)$, $(f_0)$ for $N\geq 3$, $(f_1)-(f_5)$ for $N=2$. Then,
	there exist  $\tau_0, \xi_0, \epsilon_{0}> 0$ such that $(P)_{\epsilon}$ has a ground state solution $u_\epsilon$ for all $\epsilon \in (0,\epsilon_0)$, with $\xi \geq \xi_0$ if $N=3$ and $\tau \geq \tau_0$ if $N=2$. Moreover, if $x_{\epsilon} \in \mathbb{R}^{N}$ denotes a global maximum point of $|u_{\epsilon}|$, then
	\[
	\lim_{\epsilon \rightarrow 0}A(\epsilon x_{\epsilon}) = \sup_{x \in \mathbb{R}^N}A(x).
	\]
\end{theorem}

In the proof of Theorem \ref{T1}, we will use variational methods to get a critical point for the energy function $I_{\epsilon}:H^{1}(\R^{N})\rightarrow\R$ given by 
$$
I_{\epsilon}(u)=\frac{1}{2}B(u,u)-\int_{\R^{N}} A(\epsilon x)F(u)dx,
$$
where $B:H^{1}(\R^N) \times H^{1}(\R^N) \to \R$ is the bilinear form
\begin{equation}\label{bilinear form}
B(u,v)=\int_{\R^N} (\nabla u\nabla v+V(x)uv)\, dx, \quad \forall u,v \in H^{1}(\R^N).
\end{equation}
It is well known that $I_{\epsilon}\in C^{1}(H^{1}(\R^{N}),\R)$ with
$$
I_{\epsilon}'(u)v=B(u,v)-\int_{\R^{N}} A(\epsilon x)f(u)vdx, \quad \forall u,v \in H^{1}(\mathbb{R}^N).
$$
Consequently, critical points of $I_{\epsilon}$ are precisely the weak solutions of $(P)_{\epsilon}$.

Note that the bilinear form $B$ is not positive definite, therefore it does not  induce a norm. As in \cite{SW}, there is an inner product $\langle \;\; , \;\; \rangle $  in $H^{1}(\mathbb{R}^N)$ such that 
\begin{equation} \label{PHIW}
I_\epsilon(u)=\frac{1}{2}\|u^+\|^{2}-\frac{1}{2}\|u^-\|^{2}-\int_{\R^N} A(\epsilon x)F(u)\,dx,
\end{equation}
where $\|u\|=\sqrt{\langle u, u \rangle}$ and $H^{1}(\mathbb{R}^N) = E^{+} \oplus E^-$ corresponds to the spectral decomposition of $- \Delta + V $ with respect to the
positive and negative part of the spectrum with $u = u^{+}+u^{-}$, where $u^{+} \in E^{+}$ and $u^{-} \in E^{-}$. It is well known that $B$ is positive definite on $E^{+}$, $B$ is negative definite on $E^{-}$ and the norm $\|\,\,\,\|$ is an equivalent norm to the usual norm in $H^{1}(\R^N)$, that is, there are $a,b >0 $ such that
\begin{equation} \label{equivalente}
b||u|| \leq ||u||_{H^{1}(\R^{N})}\leq a||u||,\ \ \forall\ u\in H^{1}(\R^{N}).
\end{equation}

From now on, for each $u \in H^{1}(\R^N)$,  $\hat{E}(u)$ designates  the set
\begin{equation} \label{conjuntoE}
\hat{E}(u)=E^{-}\oplus [0,+\infty)u.
\end{equation}

The plan of the paper is as follows: In Section 2 we will study the existence and concentration of solution for $N \geq 3$, while in Section 3 we will focus our attention to dimension  $N=2$.  

\vspace{0.5 cm}

\noindent \textbf{Notation:} In this paper, we use the following
notations:
\begin{itemize}
	\item  The usual norms in $H^{1}(\R^N)$ and $L^{p}(\R^N)$ will be denoted by
	$\|\;\;\;\|_{H^{1}(\R^N)}$ and $|\;\;\;|_{p}$ respectively. 	
	
	\item   $C$ denotes (possible different) any positive constant.
	
	\item   $B_{R}(z)$ denotes the open ball with center  $z$ and
	radius $R$ in $\mathbb{R}^N$.
	
	\item We say that $u_n \to u$ in $L_{loc}^{p}(\mathbb{R}^N)$  when
	$$
	u_n \to u \quad \mbox{in} \quad L^{p}(B_R(0)), \quad \forall R>0.
	$$
	\item  If $g$ is a mensurable function, the integral $\int_{\mathbb{R}^N}g(x)\,dx$ will be denoted by $\int g(x)\,dx$.	
	\item We denote $\delta_{x}$ the Dirac measure.
	\item If $\varphi\in C^{\infty}_{c}(\R^{N})$, the set $\overline{\{x\in\R^{N}\ ;\ \varphi(x)\neq0\}}$ will be denoted by $supp\varphi$.
\end{itemize}

\section{The case $N\geq 3$.}

We begin this section by studying the case where $A$ is a constant function. More precisely, we consider the following autonomous problem 
$$
\left\{\begin{array}{l}
-\Delta u+V(x)u=\lambda f(u),\quad x \in \R^{N}, \\ 
u\in H^{1}(\R^{N}),
\end{array}\right.\eqno{(AP)_{\lambda}}
$$
with $\lambda \in [A_0, +\infty)$ and $f:\R\rightarrow\R$ being of the form
$$
f(t)=\xi|t|^{q-1}t+|t|^{2^{*}-2}t \quad \forall t \in \R;
$$ 
with $\xi>0, q \in (2,2^*)$ and $2^*={2N}/{N-2}$.

Associated with $(AP)_{\lambda}$, we have the energy functional $J_{\lambda}:H^{1}(\R^{N})\rightarrow\R$ given by
$$
J_\lambda(u)=\frac{1}{2}\int (|\nabla u|^{2}+V(x)|u|^{2})\,dx-\lambda \int F(u)\,dx, 
$$
or equivalently   
$$
J_\lambda(u)=\frac{1}{2}\|u^+\|^{2}-\frac{1}{2}\|u^-\|^{2} -\lambda \int F(u)\,dx.
$$
In what follows, let us denote by $d_\lambda$ the real number defined by
\begin{equation} \label{dlambda}
d_\lambda=\inf_{u \in \mathcal{N}_{\lambda}}J_\lambda(u);
\end{equation}
where
\begin{equation} \label{Nlambda}
\mathcal{N}_{\lambda}=\left\{u\in H^{1}(\R^{N})\setminus E^{-}\ ;\ J_{\lambda}'(u)u=0\text{ and }J_{\lambda}'(u)v=0,\forall\ v\in E^{-}\right\}. 
\end{equation}

In \cite{AG}, Alves and Germano have  proved that for each $\lambda\in[A_{0},+\infty)$, the problem $(AP)_\lambda$ possesses a ground state solution $u_\lambda \in H^{1}(\R^N)$, that is, 
$$
u_\lambda \in \mathcal{N}_{\lambda}, \quad J_\lambda(u_\lambda)=d_\lambda \quad \mbox{and} \quad J_{\lambda}'(u)=0.
$$
A key point to prove the existence of the ground state $u_\lambda$ are the following informations involving $d_\lambda$:
\begin{equation} \label{positivo}
0<d_\lambda=\inf_{u\in E^{+}\setminus\{0\}}\max_{v\in\widehat{E}(u)}J_{\lambda}(u)
\end{equation}
and
\begin{equation}\label{limitacaod4}
d_{\lambda}<\frac{1}{N}\frac{S^{N/2}}{\lambda ^{\frac{N-2}{2}}}, \quad \forall \lambda \geq A_0.
\end{equation}
Here, we would like to point out that (\ref{limitacaod4}) holds for $N=3$ if $\xi$ is large enough, while for $N \geq 4$ there is no restriction on $\xi$. This fact justifies why $\xi$ must be large for $N=3$ in  Theorem \ref{T1}.

An interesting and important fact is that for each $u\in H^{1}(\R^{N})\setminus E^{-}$, $\mathcal{N}_\lambda\cap\hat{E}(u)$ is a singleton set and the element of this set is the unique global maximum of $J_{\lambda}|_{\hat{E}(u)}$, that is, there are $t^* \geq 0$ and $v^* \in E^{-}$ such that
\begin{equation} \label{maximo}
J_{\lambda}(t^*u+v^*)=\displaystyle\max_{w \in \widehat{E}(u)}J_{\lambda}(w).
\end{equation}

After the above commentaries we are ready to prove an important result involving the function $\lambda\mapsto d_{\lambda}$.

\begin{proposition}\label{continuidadec}  The function $\lambda\mapsto d_{\lambda}$ is decreasing and continuous on $[A_0,+\infty)$.
\end{proposition}
\begin{proof}
From \cite[Proposition 2.3]{AG2}, the function $\lambda\mapsto d_{\lambda}$ is decreasing, and if $\lambda_{1}\leq\lambda_{2}\leq\lambda_{3}\leq...\leq\lambda_{n}\rightarrow\lambda$ then $\lim_{n} d_{\lambda_{n}}=d_{\lambda}$. It suffices to check that $\lambda_{1}\geq\lambda_{2}\geq\lambda_{3}\geq...\geq\lambda_{n}\rightarrow\lambda$ implies $\lim_{n} d_{\lambda_{n}}=d_{\lambda}$. Let $u_{n}$ be a ground state solution of $(AP)_{\lambda_{n}}$, $t_{n}>0$ and $v_{n}\in E^{-}$ verifying 
$$
J_{\lambda}(t_{n}u_{n}+v_{n})=\max_{\widehat{E}(u_{n})}J_{\lambda}.
$$
Our goal is to show that $(u_{n})$ is bounded in $H^{1}(\mathbb{R}^N)$. First of all, note that
\begin{equation}\label{limitacaofunun}
\left(\frac{1}{2}-\frac{1}{q}\right)\int f(u_{n})u_{n}dx\leq\int\left(\frac{1}{2}f(u_{n})u_{n}-F(u_{n})\right)dx=
\end{equation}
$$
=\frac{1}{\lambda_{n}}\left(J_{\lambda_{n}}(u_{n})-\frac{1}{2}J_{\lambda_{n}}'(u_{n})u_{n}\right)=\frac{1}{\lambda_{n}} J_{\lambda_{n}}(u_{n})=\frac{1}{\lambda_{n}}d_{\lambda_{n}}\leq\frac{1}{\lambda}d_{\lambda},
$$
which proves the boundedness of $\left(\int f(u_{n})u_{n}dx\right)$. Fixing $g(t)=\chi_{[-1,1]}(t)f(t)$ and $l(t)=\chi_{[-1,1]^{c}}(t)f(t)$, we have that   
$$
g(t)+l(t)=f(t), \quad \forall t \in\R. 
$$ 
From definition of $g$ and $l$, there exists $k>0$ such that
$$
|g(t)|^{r}\leq ktf(t)\text{ and }|l(t)|^{s}\leq ktf(t), \quad \forall t \in \mathbb{R},
$$
where $r:=\frac{q+1}{q}$ and $s:=\frac{2^{*}}{2^{*}-1}$. Thus,
$$
\left|\int f(u_{n})u_{n}^{+}dx\right|\leq\int|g(u_{n})u_{n}^{+}|dx+\int|l(u_{n})u_{n}^{+}|dx\leq
$$
$$\leq \left(\int|g(u_{n})|^{r}dx\right)^{1/r}|u_{n}^{+}|_{q+1}+\left(\int|l(u_{n})|^{s}dx\right)^{1/s}|u_{n}^{+}|_{2^{*}}\leq$$
$$\leq C\left(\int f(u_{n})u_{n}dx\right)^{1/r}||u_{n}^{+}||+C\left(\int f(u_{n})u_{n}dx\right)^{1/s}||u_{n}^{+}||\leq C||u_{n}||.$$
Suppose by contradiction that $||u_{n}||\rightarrow+\infty$. Then
$$
\int\frac{f(u_{n})u_{n}^{+}}{||u_{n}||^{2}}dx\rightarrow 0.
$$
On the other hand, the equality 
$$
0=\frac{J_{\lambda_{n}}'(u_{n})u_{n}^{+}}{||u_{n}||^{2}}=\frac{||u_{n}^{+}||^{2}}{||u_{n}||^{2}}-\lambda_{n}\int\frac{f(u_{n})u_{n}^{+}}{||u_{n}||^{2}}dx
$$
leads to  
$$
\frac{||u_{n}^{+}||^{2}}{||u_{n}||^{2}}\rightarrow0.
$$
As $u_{n}\in\mathcal{N}_{\lambda_{n}}$,  it follows that $\|u_n^{-}\| \leq \|u_n^{+}\|$, and thus, 
$$
1=\frac{||u_{n}^{+}||^{2}}{||u_{n}||^{2}}+\frac{||u_{n}^{-}||^{2}}{||u_{n}||^{2}}\leq 2\frac{||u_{n}^{+}||^{2}}{||u_{n}||^{2}}\rightarrow 0,
$$
a contradiction. This shows the boundedness of $(u_{n})$. We claim that there are $(y_n) \subset \Z^N$ and $r,\eta>0$ such that 
\begin{equation}\label{desigualdadelions}
\int_{B_{r}(y_{n})}|u_{n}|^{2^{*}}dx>\eta, \quad \forall n \in \mathbb{N}.
\end{equation}
Arguing by contradiction, if the inequality does not occur, from \cite[Lemma 2.1]{RWW}, $u_{n}\rightarrow0$ in $L^{p}(\R^{N})$ for all $p\in(2,2^{*}]$, and so, $\int f(u_{n})u_{n}^{+}dx\rightarrow0$. This together with the equality below
$$
0=J_{\lambda_{n}}'(u_{n})u_{n}^{+}=||u_{n}^{+}||^{2}-\lambda_{n}\int f(u_{n})u_{n}^{+}dx.
$$
gives $||u_{n}^{+}||\rightarrow0$, which is a contradiction because $||u_{n}||\geq\sqrt{2d_{\lambda_{n}}}\geq\sqrt{2d_{\lambda_{1}}}$. Thereby (\ref{desigualdadelions}) follows.

Define $\widetilde{u}_{n}(x):=u_{n}(x+y_{n})$. By \cite[Lemma 2.1]{AG2}, $\widetilde{u}_{n}^{+}(x)=u_{n}^{+}(x+y_{n})$ and $(\widetilde{u}_{n})$ is bounded in $H^{1}(\R^{N})$. In the sequel, let us assume that for some subsequence $\widetilde{u}_{n}\rightharpoonup u$ in $H^{1}(\R^N)$. Our goal is to show that $u\neq0$. Inspired by \cite[Lemma 2.17]{AG}, let us suppose by contradiction $u=0$ and
$$
|\nabla\widetilde{u}_{n}|^{2}\rightharpoonup\mu,\ \ |\widetilde{u}_{n}|^{2^{*}}\rightharpoonup\nu \ \text{ in }\M^{+}(\R^{N}).
$$
By Concentration-Compactness Principle due to Lions \cite{lionsII}, there exist a countable set J, $(x_{i})_{i\in\text{J}} \subset \R^{N}$ and $(\mu_{i})_{i\in\text{J}},(\nu_{i})_{i\in\text{J}} \subset [0,+\infty)$ such that
$$\nu=\sum_{i\in\text{J}}\nu_{i}\delta_{x_{i}},\ \mu\geq\sum_{i\in\text{J}}\mu_{i}\delta_{x_{i}},\ \text{ and }\mu_{i}=S\nu_{i}^{2/2^{*}}.$$
We will prove that $\nu_{i}=0$ for all $i\in$J. Suppose there exists $i\in$J such that $\nu_{i}\neq0$. Then,
$$
\begin{array}{ll}
d_{\lambda} & \geq \lim_{n}d_{\lambda_{n}}=\lim_{n}\left(J_{\lambda_{n}}(u_{n})-\frac{1}{2}J_{\lambda_{n}}'(u_{n})u_{n}\right) \\ & \geq\lim_{n}\lambda_{n}\left(\frac{1}{2}-\frac{1}{2^{*}}\right)\int |u_{n}|^{2^{*}}dx \\
& =\lim_{n}\frac{\lambda_{n}}{N}\int |\widetilde{u}_{n}|^{2^{*}}dx=\frac{\lambda}{N}\sum_{j\in\text{J}}\nu_{j},\\
\end{array}
$$
which means 
\begin{equation}\label{concentracaolionsi}
d_{\lambda}\geq\frac{\lambda}{N}\sum_{j\in\text{J}}\nu_{j}.
\end{equation}
Let $\varphi_{\delta}(x):=\varphi\left(\frac{x-x_{i}}{\delta}\right)$ for all $x\in\R^{N}$ and $\delta >0$, where $\varphi\in C^{\infty}_{c}(\R^{N})$ is such that $\varphi\equiv1$ on $B_{1}(0)$, $\varphi\equiv0$ on $\R^{N}\setminus B_{2}(0)$, $0\leq\varphi\leq1$ and $|\nabla\varphi|\leq2$. Consequently $(\varphi_{\delta}\widetilde{u}_{n})$ is bounded in $H^{1}(\R^{N})$ and
$$
J_{\lambda_{n}}'(\widetilde{u}_{n})(\varphi_{\delta}\widetilde{u}_{n})=0,
$$
that is, 
$$
\int\nabla\widetilde{u}_{n}\nabla(\varphi_{\delta}\widetilde{u}_{n})dx+\int V(x)\varphi_{\delta}\widetilde{u}_{n}^{2}dx=\lambda_{n}\xi\int |\widetilde{u}_{n}|^{q+1}\varphi_{\delta}dx+\lambda_{n}\int |\widetilde{u}_{n}|^{2^{*}}\varphi_{\delta}dx.
$$
Passing to the limit as $n\rightarrow+\infty$,
$$
\int\varphi_{\delta}d\mu=\lambda\int \varphi_{\delta}d\nu.
$$
Now, taking the limit $\delta \rightarrow 0$, 
$$
\mu(x_{i})=\lambda\nu_{i}.
$$
From the fact that $\mu(x_{i})\geq\mu_{i}$, we derive 
$$
S\nu_{i}^{2/2^{*}}=\mu_{i}\leq\mu(x_{i})=\lambda \nu_{i},
$$
and so 
$$
S^{N/2}\leq \lambda^{N/2}\nu_{i}.
$$
Consequently, 
\begin{equation}\label{concentracaolionsii}
\frac{\lambda}{N}\nu_{i}\geq\frac{1}{N}\frac{S^{N/2}}{\lambda^{\frac{N-2}{2}}}.
\end{equation}
From (\ref{concentracaolionsi}) and (\ref{concentracaolionsii}),  
$$
d_{\lambda}\geq\frac{1}{N}\frac{S^{N/2}}{\lambda^{\frac{N-2}{2}}},
$$
contrary to (\ref{limitacaod4}). From this,  $\nu_{i}=0$ for all $i\in$J and $\widetilde{u}_{n}\rightarrow 0$ in $L^{2^{*}}_{loc}(\R^{N})$, which contradicts (\ref{desigualdadelions}). This permit us to conclude that $u\neq 0$.

\begin{claim} \label{C1}
If $u^{+}=0$, then $u^{-}=0$.
\end{claim}
In fact, if $u^{+}=0$, 
$$
\int f(u)u^{-}dx=\int f(u)u^{+}dx+\int f(u)u^{-}dx=\int f(u)udx\geq0.
$$
On the other hand, letting $n\rightarrow+\infty$ in the equality below 
$$
0=J_{\lambda_{n}}(\widetilde{u}_{n})u^{-}=B(\widetilde{u}_{n},u^{-})-\lambda_{n}\int f(\widetilde{u}_{n})u^{-}dx
$$
we find
$$
-||u^{-}||^{2}=B(u,u^{-})=\lambda\int f(u)u^{-}dx\geq 0,
$$
thereby showing that $u^{-}=0$. 

The Claim \ref{C1} implies that $u^{+}\neq0$, because $u\neq0$ and $u=u^{+}+u^{-}$. Define ${\cal V}:=\{\widetilde{u}_{n}^{+}\}_{n\in\N}$. Since $\widetilde{u}_{n}^{+}\rightharpoonup u^{+}\neq0$, then $0\notin\overline{{\cal V}}^{\sigma(H^{1}(\R^{N}),H^{1}(\R^{N})')}$ and ${\cal V}$ is bounded in $H^{1}(\R^{N})$. Applying \cite[Lemma 2.2]{AG2}, there exists $R>0$ such that
\begin{equation}\label{desigualdadejl}
J_{\lambda}\leq0\text{ on }\widehat{E}(u)\setminus B_{R}(0),\ \text{ for all }u\in {\cal V}.
\end{equation}
Setting $\widetilde{v}_{n}(x):=v_{n}(x+y_{n})$,  
\begin{equation}\label{desigualdadejl2}
J_{\lambda}(t_{n}\widetilde{u}_{n}+\widetilde{v}_{n})=J_{\lambda}(t_{n}u_{n}+v_{n})\geq d_{\lambda}>0.
\end{equation}
By (\ref{desigualdadejl}) and (\ref{desigualdadejl2}), $||t_{n}\widetilde{u}_{n}+\widetilde{v}_{n}||\leq R$ for all $n\in\N$. As  $||t_{n}u_{n}+v_{n}||=||t_{n}\widetilde{u}_{n}+\widetilde{v}_{n}||$, $(t_{n}u_{n}+v_{n})$ is also bounded in $H^{1}(\R^{N})$ and 
$$
\begin{array}{ll}
d_{\lambda} & \leq J_{\lambda}(t_{n}u_{n}+v_{n})=(\lambda_{n}-\lambda)\int F(t_{n}u_{n}+v_{n})dx+J_{\lambda_{n}}(t_{n}u_{n}+v_{n})\leq \\
& \leq o_{n}+J_{\lambda_{n}}(u_{n})=o_{n}+d_{\lambda_{n}}\leq o_{n}+d_{\lambda},
\end{array}
$$
from where it follows that $\lim_{n} d_{\lambda_{n}}=d_{\lambda}$.
\end{proof}

\subsection{Existence of ground state for problem $(P)_\epsilon$.}

In the sequel, we fix
$$
\M_{\epsilon}:=\{u\in H^{1}(\R^{N})\setminus E^{-}\ ;\ I_{\epsilon}'(u)u=I_{\epsilon}'(u)v=0,\text{ for all }v\in E^{-}\}
$$
and 
$$
c_{\epsilon}=\inf_{\M_{\epsilon}}I_{\epsilon}.
$$
By using the same arguments found in \cite{AG}, it follows that $c_{\epsilon}>0$, and for each $u\in H^{1}(\R^{N})\setminus E^{-}$, there exist $t\geq0$ and $v\in E^{-}$ verifying
$$
I_{\epsilon}(tu+v)=\max_{\widehat{E}(u)}I_{\epsilon}\ \text{ and }\ \{tu+v\}=\M_{\epsilon}\cap\widehat{E}(u).
$$
The same idea of \cite[Lemma 2.6]{AG} proves that 
\begin{equation}\label{umaislimitadainferiormente}
||u^{+}||^{2}\geq2c_{\epsilon},\ \ \text{ for all }u\in\M_{\epsilon} \quad \mbox{and} \quad \epsilon>0.
\end{equation}

In what follows, without loss of generality we assume that
$$
A(0)=\max_{x \in \R^N}A(x).
$$

Our first result in this section establishes an important relation involving the levels $c_\epsilon$ and $c_0$. 
\begin{lemma} \label{Lema1}
The limit $\lim_{\epsilon\rightarrow0}c_{\epsilon}=c_{0}$ holds. Moreover, let  $w_{0}$ be a ground state solution of the problem $(P)_{0}$, $t_{\epsilon}\geq0$ and $v_{\epsilon}\in E^{-}$ such that $t_{\epsilon}w_{0}+v_{\epsilon}\in\M_{\epsilon}$. Then
$$
t_{\epsilon}\rightarrow1\ \text{ and }\ v_{\epsilon}\rightarrow 0 \quad \mbox{as} \quad \epsilon\rightarrow 0.
$$ 
\end{lemma}
\begin{proof}
See \cite[Lemmas 3.1 and 3.3]{AG2}. 
\end{proof}

\begin{corollary}\label{limitacaoce}
There exists $\epsilon_{0}>0$ such that 
$$
c_{\epsilon}<d_{A_{\infty}}\ \text{ and }\ c_{\epsilon}<\frac{S^{N/2}}{NA(0)^{\frac{N-2}{2}}}, \quad \forall \epsilon \in (0, \epsilon_0).
$$
\end{corollary}
\begin{proof} Since $c_0 < d_{A_\infty}$ and 
$$	
c_{0}<\frac{S^{N/2}}{NA(0)^{\frac{N-2}{2}}} \quad (\mbox{see} \,\,\, (\ref{limitacaod4}) ),
$$
the corollary is an immediate  consequence of Lemma \ref{Lema1}.
\end{proof}

\vspace{0.5 cm}

The next result is essential to show the existence of ground state solution of $(P)_{\epsilon}$ for $\epsilon$ small enough. Since it follows as in \cite[Proposition 2.16]{AG}, we omit its proof. 

\begin{proposition} \label{sequenciaPS}
There exists a bounded sequence $(u_{n}) \subset \M_{\epsilon}$ such that $(u_{n})$ is $(PS)_{c_{\epsilon}}$ for $I_{\epsilon}$.
\end{proposition}

The following result is the main result this section

\begin{theorem}\label{existenciapontocritico}
The problem $(P)_{\epsilon}$ has a ground state solution for all $\epsilon\in(0,\epsilon_{0})$, where $\epsilon_0>0$ was given in Corollary \ref{limitacaoce}.
\end{theorem}
\begin{proof} Let  $(u_{n}) \subset \M_{\epsilon}$ be the $(PS)_{c_{\epsilon}}$ sequence for $I_{\epsilon}$ given in  Proposition \ref{sequenciaPS}. Then, there exist $(z_{n}) \subset \Z^{N}$ and $\eta,r>0$ such that
\begin{equation}\label{lionsun}
\int_{B_{r}(z_{n})}|u_{n}|^{2^{*}}dx>\eta, \quad \forall n \in \mathbb{N}. 
\end{equation}
In fact, otherwise, by \cite[Lemma 2.1]{RWW}, $u_{n}\rightarrow 0$ in $L^{p}(\R^{N})$ for all $p\in(2,2^{*}]$. Then, 
$$
||u_{n}^{+}||^{2}=\int A(\epsilon x)f(u_{n})u_{n}^{+}dx\rightarrow0,
$$
which is a contradiction with (\ref{umaislimitadainferiormente}), and (\ref{lionsun}) is proved.  
\begin{claim}\label{znlimitada}
The sequence $(z_{n})$ is bounded in $\R^{N}$.
\end{claim}
Arguing by contradiction, suppose $|z_{n}|\rightarrow+\infty$ and define $w_{n}(x):=u_{n}(x+z_{n})$. Then $(w_{n})$ is bounded, and for some subsequence, $w_{n}\rightharpoonup w$ in $H^{1}(\mathbb{R}^N)$. Our goal is to prove that $w\neq0$. Suppose $w=0$ and 
$$
|\nabla w_{n}|^{2}\rightharpoonup\mu,\ \ |w_{n}|^{2^{*}}\rightharpoonup\nu,\ \text{ in }\M^{+}(\R^{N}).
$$
By Concentration-Compactness Principle due to Lions \cite{lionsII}, there exist a countable set J, $(x_{i})_{i\in\text{J}} \subset \R^{N}$ and $(\mu_{i})_{i\in\text{J}},(\nu_{i})_{i\in\text{J}} \subset [0,+\infty)$ satisfying
$$
\nu=\sum_{i\in\text{J}}\nu_{i}\delta_{x_{i}},\ \ \mu\geq\sum_{i\in\text{J}}\mu_{i}\delta_{x_{i}},\ \text{ and} \quad \mu_{i}=S\nu_{i}^{2/2^{*}}.
$$
Next, we are going to prove that $\nu_{i}=0$ for all $i\in$J. Suppose that there exists $i\in$J such that $\nu_{i}\neq0$. Note that
$$
\begin{array}{l}
c_{\epsilon}=\lim_{n}\left(I_{\epsilon}(u_{n})-\frac{1}{2}I_{\epsilon}'(u_{n})u_{n}\right)\geq\frac{1}{N}\lim_{n}\int A(\epsilon x)|u_{n}|^{2^{*}}dx= \\
=\frac{1}{N}\lim_{n}\int A(\epsilon x+\epsilon z_{n})|w_{n}|^{2^{*}}dx\geq\frac{1}{N}\lim_{n}\int_{B_{\delta}(x_{i})} A(\epsilon x+\epsilon z_{n})|w_{n}|^{2^{*}}dx=\\
=\frac{1}{N}\lim_{n}\int_{B_{\delta}(x_{i})} (A(\epsilon x+\epsilon z_{n})-A_{\infty})|w_{n}|^{2^{*}}dx+\frac{1}{N}\lim_{n}\int_{B_{\delta}(x_{i})} A_{\infty}|w_{n}|^{2^{*}}dx\geq\\
\geq \frac{1}{N}\int A_{\infty}\varphi_{\delta/2}(x)d\nu,
\end{array}
$$
where $\varphi_{\delta}(x)=\varphi\left(\frac{x-x_{i}}{\delta}\right)$,  and $\varphi\in C^{\infty}_{c}(\R^{N})$ satisfies $0\leq\varphi\leq1$, $|\nabla\varphi|\leq2$, $\varphi\equiv1$ on $B_{1}(0)$ and $\varphi\equiv0$ on $\R^{N}\setminus B_{2}(0)$. 

By Dominated Convergence Theorem, 
$$
\lim_{\delta\rightarrow0}\int A_{\infty}\varphi_{\delta/2}(x)d\nu=A_{\infty}\nu_{i},
$$
thus
\begin{equation}\label{desigualdadece}
c_{\epsilon}\geq\frac{1}{N}A_{\infty}\nu_{i}.
\end{equation}
On the other hand, by a simple calculus, $(\varphi_{\delta}w_{n})$ is bounded in $H^{1}(\R^{N})$. Setting $\varphi_{\delta,n}(x):=\varphi_{t}(x-z_{n})$, 
$$
||\varphi_{\delta,n}u_{n}||=||\varphi_{\delta}w_{n}||, \quad \forall n \in \mathbb{N}
$$
and so, 
$$
I_{\epsilon}'(u_{n})(\varphi_{\delta,n}u_{n})\rightarrow 0,
$$
or equivalently
$$
\begin{array}{l}
\int|\nabla w_{n}|^{2}\varphi_{\delta}dx+\int(\nabla w_{n}\nabla\varphi_{\delta})w_{n}dx+\int V(x)\varphi_{\delta}w_{n}^{2}dx-\\
-\int A(\epsilon x+\epsilon z_{n})|w_{n}|^{q+1}\varphi_{\delta}dx-\int A(\epsilon x+\epsilon z_{n})|w_{n}|^{2^{*}}\varphi_{\delta}dx\rightarrow0
\end{array}
$$
Taking the limit $n\rightarrow+\infty$, and after  $\delta \to 0$, we obtain
$$
\mu(x_{i})=A_{\infty}\nu_{i}.
$$
Since $S\nu_{i}^{2/2^{*}}\leq\mu(x_{i})$, it follows that
\begin{equation}\label{concentracaolions}
S^{N/2}\leq A_\infty^{\frac{N}{2}}\nu_{i} \leq A(0)^{\frac{N-2}{2}}A_{\infty}\nu_{i}.
\end{equation}
By (\ref{desigualdadece}) and (\ref{concentracaolions}), 
$$
c_{\epsilon}\geq\frac{S^{N/2}}{NA(0)^{\frac{N-2}{2}}},
$$
which is impossible by Corollary \ref{limitacaoce}. Consequently $\nu_{i}=0$ for all $i\in$J, which means $w_{n}\rightarrow 0$ in $L^{2^{*}}_{loc}(\R^{N})$, contrary to (\ref{lionsun}). From this, $w\neq0$.

Now, consider $\psi\in H^{1}(\R^{N})$ and $\psi_{n}(x):=\psi(x+z_{n})$. Then, 
$$
o_n(1)=I_{\epsilon}'(u_{n})\psi_{n}=B(u_{n},\psi_{n})-\int A(\epsilon x)f(u_{n})\psi_{n}dx
$$
or equivalently
$$
o_{n}=B(w_{n},\psi)-\int A(\epsilon x+\epsilon z_{n})f(w_{n})\psi dx.
$$
Taking the limit $n\rightarrow+\infty$, $J_{A_{\infty}}'(w)\psi=0$. As $\psi\in H^{1}(\R^{N})$ is arbitrary, $w$ is a critical point of $J_{A_{\infty}}$, and thus, by Fatou's Lemma
$$
\begin{array}{ll}
d_{A_{\infty}}& \leq  J_{A_{\infty}}(w)=J_{A_{\infty}}(w)-\frac{1}{2}J_{A_{\infty}}'(w)w\\
& =\int A_{\infty}\left(\frac{1}{2}f(w)w-F(w)\right)dx\\
& \leq\liminf_{n}\int A(\epsilon x+\epsilon z_{n})\left(\frac{1}{2}f(w_{n})w_{n}-F(x,w_{n})\right)dx\\
& =\liminf_{n}\int A(\epsilon x)\left(\frac{1}{2}f(u_{n})u_{n}-F(u_{n})\right)dx\\
& =\lim_{n}\left(I_{\epsilon}(u_{n})-\frac{1}{2}I_{\epsilon}'(u_{n})u_{n}\right)=c_{\epsilon}<d_{A_{\infty}},\\
\end{array}
$$
which is absurd. Thereby $(z_{n})$ is bounded in $\R^{N}$, and the claim follows. 

Consider $R>0$ such that $B_{r}(z_{n})\subset B_{R}(0)$. By (\ref{lionsun}), 
$$
\int_{B_{R}(0)}|u_{n}|^{2^{*}}dx>\eta, \quad \forall n\in\N. 
$$
By considering that $u_{n}\rightharpoonup u$ and proceeding as in Claim \ref{znlimitada},  $u\neq0$. Since $u$ is a nontrivial critical point for $I_{\epsilon}$, we must have  $I_{\epsilon}(u)\geq c_{\epsilon}$. On the other hand, by Fatou's Lemma, 
$$
\begin{array}{ll}
c_{\epsilon}& =\lim_{n}\left(I_{\epsilon}(u_{n})-\frac{1}{2}I_{\epsilon}'(u_{n})u_{n}\right)=\lim_{n} \int A(\epsilon x)\left(\frac{1}{2}f(u_{n})u_{n}-F(u_{n})\right)dx\\ 
& \geq \int A(\epsilon x)\left(\frac{1}{2}f(u)u-F(u)\right)dx=I_{\epsilon}(u)-\frac{1}{2}I_{\epsilon}'(u)u=I_{\epsilon}(u).
\end{array}
$$
This proves that $u$ is a ground state solution of $(P)_{\epsilon}$ for all $\epsilon \in (0, \epsilon_0)$.
\end{proof}

\subsection{Concentration of the solutions.}
In  what follows, we consider the set 
$$
\mathcal{A}:=\{z\in\R^{N}\ ;\ A(z)=A(0)\},
$$ 
and a sequence $(\epsilon_{n}) \subset (0,\epsilon_{0})$ with $\epsilon_{n}\rightarrow0$ as $n\rightarrow+\infty$. Moreover, we fix $u_{n}\in H^{1}(\R^{N})$ satisfying 
$$
I_{n}(u_{n})=c_{n}\ \ \text{ and }\ \ I_{n}'(u_{n})=0,
$$
where $I_{n}:=I_{\epsilon_{n}}$ and $c_{n}:=c_{\epsilon_{n}}$. Using the same arguments explored in \cite[Lemma 2.6]{AG}, 
\begin{equation}\label{limitadoinferiormente}
||u_{n}^{+}||^{2}\geq2c_{n}\geq 2c_0, \quad \forall n \in \mathbb{N}.
\end{equation}

\begin{lemma}
The sequence $(u_{n})$ is bounded in $H^{1}(\R^{N})$.
\end{lemma}
\begin{proof}
See \cite[Lemma 2.10]{AG}.
\end{proof}

\begin{lemma}
There exist $(y_{n}) \subset \Z^{N}$ and $r,\eta>0$ such that
$$
\int_{B_{r}(y_{n})}|u_{n}|^{2^{*}}dx>\eta, \quad \forall n\in\N.
$$

\end{lemma}
\begin{proof}
Suppose the lemma were false. Then, by \cite[Lemma 2.1]{RWW}, $u_{n}\rightarrow 0$ in $L^{p}(\R^{N})$ for all $p\in(2,2^{*}]$, and so, 
$$
\int A(\epsilon_{n}x)f(u_{n})u_{n}^{+}dx\rightarrow 0.
$$
As $I_{n}'(u_{n})u_{n}^{+}=0$, it follows that $||u_{n}^{+}||^{2}\rightarrow0 $, a contradiction. This proves the lemma.
\end{proof}

\vspace{0.5 cm}

In the sequel, we fix $v_{n}(x):=u_{n}(x+y_{n})$ for all $x\in\R^{N}$ and for all $n\in\N$. Thereby, for some subsequence, we can assume that  $v_{n}\rightharpoonup v$ in $H^{1}(\R^2)$. It is very important to point out that only one of the cases below holds for some subsequence:
$$ 
\epsilon_{n}y_{n}\rightarrow z\in\R^{N}
$$
or
$$
|\epsilon_{n}y_{n}|\rightarrow+\infty.
$$
For this reason, we will consider a subsequence of $(\epsilon_{n})$ such that one of the above conditions holds. Have this in mind, let us denote 
$$
A_{z}:=\left\{\begin{array}{l} A(z),\text{ if the condition (1) holds} \\ A_{\infty},\text{ if the condition (2) holds}.\end{array}\right.
$$
Since $A$ is continuous, it follows that $|A(\epsilon_{n}x+\epsilon_{n}y_{n})-A_{z}|\rightarrow0$ uniformly with respect to $x$ on bounded Borel sets $B\subset\R^{N}$. Consequently
\begin{equation}\label{igualdadeconjuntoslimitados}
\lim\int_{B}A(\epsilon_{n}x+\epsilon_{n}y_{n})|v_{n}|^{2^{*}}\varphi dx=\lim \int_{B}A_{z}|v_{n}|^{2^{*}}\varphi dx,
\end{equation}
for each $\varphi\in L^{\infty}(\R^{N})$.

By using (\ref{igualdadeconjuntoslimitados}) and applying the same idea of Claim \ref{znlimitada}, we see that $v\neq0$.

\begin{lemma}\label{enynlimitado}
The sequence $(\epsilon_{n}y_{n})$ is bounded in $\R^{N}$. Moreover, $J'_{A(0)}(v)=0$ and if $\epsilon_{n}y_{n}\rightarrow z\in\R^{N}$, then $z\in\mathcal{A}$.  
\end{lemma}
\begin{proof}
First of all, we will prove that $(\epsilon_{n}y_{n})$ is bounded. Suppose that $|\epsilon_{n}y_{n}|\rightarrow+\infty$. Consider $\psi\in C^{\infty}_{c}(\R^{N})$ and  $\psi_{n}(x):=\psi(x-y_{n})$. Since $I_{n}'(u_{n})\psi_{n}=0$ for all $n\in\N$, then
$$
\int\nabla u_{n}\nabla\psi_{n}+V(x)u_{n}\psi_{n}dx=\int A(\epsilon_{n}x)f(u_{n})\psi_{n}dx,
$$
or equivalently
$$
\int\nabla v_{n}\nabla\psi+V(x)v_{n}\psi dx=\int A(\epsilon_{n}x+\epsilon_{n}y_{n})f(v_{n})\psi dx.
$$
Taking the limit $n\rightarrow+\infty$, we derive
$$
\int\nabla v\nabla\psi+V(x)v\psi dx=\int A_{\infty}f(v)\psi dx,
$$
thereby showing that $J_{A_{\infty}}'(v)=0$. As  $v\neq0$, the Fatou's Lemma yields
$$
\begin{array}{ll}
d_{A_{\infty}}& \leq J_{A_{\infty}}(v)=J_{A_{\infty}}(v)-\frac{1}{2}J_{A_{\infty}}'(v)v=\int A_{\infty}\left(\frac{1}{2}f(v)v-F(v)\right)dx \\
& \leq\liminf_{n}\int A(\epsilon_{n}x+\epsilon_{n}y_{n})\left(\frac{1}{2}f(v_{n})v_{n}-F(v_{n})\right)dx \\
& =\liminf_{n}\int A(\epsilon_{n}x)\left(\frac{1}{2}f(u_{n})u_{n}-F(u_{n})\right)dx\\
& =\liminf_{n}\left(I_{n}(u_{n})-\frac{1}{2}I_{n}'(u_{n})u_{n}\right)=\lim_{n} c_{n}=c_{0},
\end{array}
$$
which is absurd, because $c_{0}<d_{A_{\infty}}$. This completes the proof that $(\epsilon_{n}y_{n})$ is bounded in $\R^{N}$. Now suppose $\epsilon_{n}y_{n}\rightarrow z\in\R^{N}$. Arguing as above,  
$$
\int\nabla v\nabla\psi+V(x)v\psi dx=\int A(z)f(v)\psi dx,\ \ \psi\in C^{\infty}_{c}(\R^{N}),
$$
and so, $J_{A(z)}'(v)=0$. Hence, 
$$
d_{A(z)}\leq J_{A(z)}(v)-\frac{1}{2}J_{A(z)}'(v)v\leq\liminf_{n}\left(I_{n}(u_{n})-\frac{1}{2}I_{n}'(u_{n})u_{n}\right)=c_{0}=d_{A(0)}.
$$
Since $\lambda\mapsto d_{\lambda}$ is decreasing and $d_{A(z)}\leq d_{A(0)}$, we must have  $A(0)\leq A(z)$. From the fact that $A(0)=\max_{x\in\R^{N}}A(x)$, we obtain $A(0)=A(z)$, or equivalently, $z\in\mathcal{A}$. Moreover, we also have $J_{A(0)}'(v)=J_{A(z)}'(v)=0$.
\end{proof}

From now on we consider $\epsilon_{n}y_{n}\rightarrow z$ with $z\in\mathcal{A}$. Our goal is to prove that $v_{n}\rightarrow v$ in $H^{1}(\R^{N})$ and $v_{n}(x)\rightarrow0$ as $|x|\rightarrow+\infty$ uniformly in $n$. Have this in mind, we need of the following estimate

\begin{proposition}\label{funcaodominada}
There exists $h\in L^{1}(\R^{N})$ and a subsequence of $(v_n)$ such that
$$
|f(v_{n}(x))v_{n}(x)|\leq h(x), \quad \forall x\in\R^{N} \quad \mbox{and} \quad n \in \mathbb{N}.
$$

\end{proposition}
\begin{proof}
By Fatou's Lemma,
$$
\begin{array}{cl}
d_{A(0)} & \leq J_{A(0)}(v)=J_{A(0)}(v)-\frac{1}{2}J_{A(0)}'(v)v\\
 & =\int A(0)\left(\frac{1}{2}f(v)v-F(v)\right)dx  \\
 & =\int A(z)\left(\frac{1}{2}f(v)v-F(v)\right)dx \\
 & \leq\liminf_{n}\int A(\epsilon_{n}x+\epsilon_{n}y_{n})\left(\frac{1}{2}f(v_{n})v_{n}-F(v_{n})\right)dx \\
 & \leq\limsup_{n}\int A(\epsilon_{n}x+\epsilon_{n}y_{n})\left(\frac{1}{2}f(v_{n})v_{n}-F(v_{n})\right)dx \\ 
 &=\limsup_{n}\int A(\epsilon_{n}x)\left(\frac{1}{2}f(u_{n})u_{n}-F(u_{n})\right)dx \\
 & =\limsup_{n}\left(I_{n}(u_{n})-\frac{1}{2}I_{n}'(u_{n})u_{n}\right)=\lim_{n} c_{n}=c_{0}=d_{A(0)},
 \end{array}
 $$
from where it follows that
 $$
 \lim_{n}\int A(\epsilon_{n}x+\epsilon_{n}y_{n})\left(\frac{1}{2}f(v_{n})v_{n}-F(v_{n})\right)dx 
 =\int A(z)\left(\frac{1}{2}f(v)v-F(v)\right)dx.
 $$
 Since 
 $$
 A(\epsilon_{n}x+\epsilon_{n}y_{n})\left(\frac{1}{2}f(v_{n})v_{n}-F(v_{n})\right)\geq0
 $$
 and 
 $$
 A(\epsilon_{n}x+\epsilon_{n}y_{n})\left(\frac{1}{2}f(v_{n})v_{n}-F(v_{n})\right)\to A(z)\left(\frac{1}{2}f(v)v-F(v)\right) \quad \mbox{a.e. in} \quad \mathbb{R}^N,
 $$
 we can ensure that 
  $$
  A(\epsilon_{n}x+\epsilon_{n}y_{n})\left(\frac{1}{2}f(v_{n})v_{n}-F(v_{n})\right)\rightarrow A(z)\left(\frac{1}{2}f(v)v-F(v)\right) \quad \mbox{in} \quad L^{1}(\mathbb{R}^N).
  $$
 Thereby, there exists $\widetilde{h}\in L^{1}(\R^{N})$ such that, for some subsequence,
  $$
  A(\epsilon_{n}x+\epsilon_{n}y_{n})\left(\frac{1}{2}f(v_{n})v_{n}-F(v_{n})\right)\leq\widetilde{h}(x), \quad \forall n \in \mathbb{N}.
  $$
As
  $$
  \left(\frac{1}{2}-\frac{1}{q+1}\right)\left(\inf_{\R^{N}}A\right) f(v_{n})v_{n}\leq A(\epsilon_{n}x+\epsilon_{n}y_{n})\left(\frac{1}{2}f(v_{n})v_{n}-F(v_{n})\right),
  $$
  we get the desired result. 
\end{proof}

An immediate consequence of the last proposition is the following corollary  
\begin{corollary}
$v_{n}\rightarrow v$ in $L^{2^{*}}(\R^{N})$.
\end{corollary}
\begin{proof}
The result follows because $|v_{n}|^{2^{*}}\leq f(v_{n})v_{n}$ for all $n \in \mathbb{N}$ and $v_n(x) \to v(x)$ a.e. in $\mathbb{R}^N$.
\end{proof}

\vspace{0.5 cm}

Our next result establishes a key estimate involving the $L^{\infty}$ norm on balls for the sequence $(v_n)$. To this end, we fix $v_{n,+}=\max\{0,v_{n}\}$ and $v_{n,-}=\max\{0,-v_{n}\}$.
\begin{lemma} \label{Lema2.13}
There exist $R>0$ and $C>0$ such that
\begin{equation}\label{regularidade}
|v_{n}|_{L^{\infty}(B_{R}(x))}\leq C|v_{n}|_{L^{2^{*}}(B_{2R}(x))}, \quad \forall n\in\N\ \ \text{ and }\ \ \forall x\in\R^{N}.
\end{equation}
Hence, as $(v_n)$ is a bounded sequence in $L^{2^*}(\R^N)$, $v_{n}\in L^{\infty}(\R^{N})$ and there is $C>0$ such that 
\begin{equation}\label{regularidade0}
|v_{n}|_{\infty}\leq C, \quad \forall n\in\N.
\end{equation}
\end{lemma}
\begin{proof}
It suffices to check that
$$
|v_{n,+}|_{L^{\infty}(B_{R}(x))}\leq C|v_{n,+}|_{L^{2^{*}}(B_{2R}(x))},
$$
for all $n\in\N$ and $x\in\R^{N}$, because similar reasoning proves
$$
|v_{n,-}|_{L^{\infty}(B_{R}(x))}\leq C|v_{n,-}|_{L^{2^{*}}(B_{2R}(x))},
$$
for all $n\in\N$ and $x\in\R^{N}$. To begin with, we recall that there exist $c_{1},c_{2}>0$ satisfying 
\begin{equation}\label{desigualdadesobref}
|f(t)|\leq c_{1}|t|+c_{2}|t|^{2^{*}-1},\ \ \text{for all } t \in \R
\end{equation}
and that $v_{n}$ is a solution for the problem
$$\left\{\begin{array}{l}
-\Delta v_{n}+V(x)v_{n}=A(\epsilon_{n}x+\epsilon_{n}y_{n})f(v_{n})\ \quad \mbox{in} \quad \R^{N}, \\
v_{n}\in H^{1}(\R^{N}).
\end{array}\right.$$
We consider $\eta\in C^{\infty}_{c}(\R^{N})$, $L>0$ and $\beta>1$ arbitrary, and define $z_{L,n}:=\eta^{2}v_{L,n}^{2(\beta-1)}v_{n,+}$ and $w_{L,n}:=\eta v_{n,+}v_{L,n}^{\beta-1}$ where $v_{L,n}=\min\{v_{n,+},L\}$. Applying $z_{L,n}$ as a test function, we find
\begin{equation}\label{desigualdade1}
\int\eta^{2}v_{L,n}^{2(\beta-1)}|\nabla v_{n,+}|^{2}dx\leq|A|_{\infty}\int|f(v_{n})|\eta^{2}v_{L,n}^{2(\beta-1)}v_{n,+}dx-
\end{equation}
$$
-\int V(x)v_{n}v_{L,n}^{2(\beta-1)}\eta^{2}v_{n,+}dx-2\int(\nabla v_{n}\nabla\eta)\eta v_{L,n}^{2(\beta-1)}v_{n,+}dx.
$$
Since 
\begin{equation}\label{desigualdade2}
\left|\int v_{L,n}^{2(\beta-1)}(v_{n,+}\nabla\eta)(\eta\nabla v_{n})dx\right|\leq C\int v_{L,n}^{2(\beta-1)}v_{n,+}^{2}|\nabla\eta|^{2}dx+
\end{equation}
$$
+\frac{1}{4}\int v_{L,n}^{2(\beta-1)}\eta^{2}|\nabla v_{n,+}|^{2}dx,
$$
combining (\ref{desigualdadesobref}), (\ref{desigualdade1}) and (\ref{desigualdade2}), we obtain
\begin{equation}\label{desigualdade3}
\int \eta^{2}v_{L,n}^{2(\beta-1)}|\nabla v_{n,+}|^{2}dx\leq C\int|v_{n,+}|^{2}\eta^{2}v_{L,n}^{2(\beta-1)}dx+
\end{equation}
$$+C\int|v_{n}|^{2^{*}}\eta^{2}v_{L,n}^{2(\beta-1)}dx+C\int v_{L,n}^{2(\beta-1)}v_{n,+}^{2}|\nabla\eta|^{2}dx$$
where $C>0$ is independently of $\beta>1, \eta\in C^{\infty}_{c}(\R^{N})$ and $L>0$.

On the other hand, since $H^{1}(\R^{N}) \hookrightarrow D^{1,2}(\R^{N})\hookrightarrow L^{2^{*}}(\R^{N})$, 
\begin{equation}\label{desigualdade4}
|w_{L,n}|_{2^{*}}^{2}\leq C\int|\nabla w_{L,n}|^{2}dx\leq C\int|\nabla\eta|^{2}v_{L,n}^{2(\beta-1)}v_{n,+}^{2}dx+
\end{equation}
$$
C\int\eta^{2} v_{L,n}^{2(\beta-1)}|\nabla v_{n,+}|^{2}dx+C\int\eta^{2}|\nabla v_{L,n}^{(\beta-1)}|^{2}v_{n,+}^{2}dx,
$$
and thus
\begin{equation}\label{desigualdade5}
|w_{L,n}|^{2}_{2^{*}}\leq C\beta^{2}\left(\int|\nabla\eta|^{2}v_{L,n}^{2(\beta-1)}v_{n,+}^{2}dx+\int\eta^{2}v_{L,n}^{2(\beta-1)}|\nabla v_{n,+}|^{2}dx\right).
\end{equation}
Then, from (\ref{desigualdade3}) and (\ref{desigualdade5}), 
\begin{equation}\label{desigualdade6}
|w_{L,n}|^{2}_{2^{*}}\leq C\beta^{2}\left(\int|v_{n,+}|^{2}\eta^{2}v_{L,n}^{2(\beta-1)}dx+\right.
\end{equation}
$$
\left.+\int|v_{n}|^{2^{*}}\eta^{2}v_{L,n}^{2(\beta-1)}dx+\int v_{L,n}^{2(\beta-1)}v_{n,+}^{2}|\nabla\eta|^{2}dx\right),
$$
where $C>0$ is independently of $n\in\N$, $\beta>1$, $L>0$ and $\eta\in C^{\infty}_{c}(\R^{N})$.

\begin{claim}\label{afirmacaoregularidade}
There exists $R>0$ such that
$$\sup_{n\in\N, x\in\R^{N}}\int_{B_{3R}(x)} v_{n,+}^{\frac{{2^{*}}^{2}}{2}}dx<+\infty.$$
\end{claim}
In fact, fix $\beta_{0}:=\frac{2^{*}}{2}$. By using the limit $v_{n}\rightarrow v$ in $L^{2^{*}}(\R^{N})$, we can fix $R>0$ sufficiently small verifying 
\begin{equation}
C\beta_{0}^{2}\left(\int_{B_{4R}(x)}v_{n,+}^{2^{*}}dx\right)^{\frac{2^{*}-2}{2}}<\frac{1}{2},\ \text{ for all }n\in\N\text{ and }x\in\R^{N},
\end{equation}
where $C$ is given in (\ref{desigualdade6}). On the other hand, consider $\eta_{x}\in C^{\infty}_{c}(\R^{N},[0,1])$ such that $\eta_{x}\equiv1$ on $B_{3R}(x)$, $\eta_{x}\equiv0$ on $\R^{N}\setminus B_{4R}(x)$ and $x\mapsto||\nabla\eta_{x}||_{\infty}$ is a constant function. Then,
$$\int v_{n,+}^{2^{*}}\eta_{x}^{2}v_{L,n}^{2(\beta_{0}-1)}=\int v_{n,+}^{2^{*}}\eta_{x}^{2}v_{L,n}^{2^{*}-2}=\int_{B_{4R}(x)}\left(v_{n,+}^{2}\eta_{x}^{2}v_{L,n}^{2^{*}-2}\right)v_{n,+}^{2^{*}-2}dx\leq$$
$$\leq\left(\int\left(v_{n,+}\eta_{x}v_{L,n}^{\frac{2^{*}-2}{2}}\right)^{2^{*}}dx\right)^{\frac{2}{2^{*}}}\left(\int_{B_{4R}(x)}v_{n,+}^{2^{*}}dx\right)^{\frac{2^{*}-2}{2}}\leq\frac{1}{2C\beta_{0}^{2}}|w_{L,n}|^{2}_{2^{*}}$$
Applying (\ref{desigualdade6}) with $\eta=\eta_{x}$ and $\beta=\beta_{0}$, we get
$$
|w_{L,n}|^{2}_{2^{*}}\leq C\beta_{0}^{2}\left(\int\eta_{x}^{2}v_{n,+}^{2^{*}}dx+\frac{1}{2C\beta_{0}}|w_{L,n}|^{2}_{2^{*}}+\int v_{n,+}^{2^{*}}|\nabla\eta_{x}|^{2}dx\right),
$$
which leads to 
$$
|w_{L,n}|^{2}_{2^{*}}\leq C\beta_{0}^{2}\left(1+||\nabla\eta_{x}||_{\infty}\right)\int v_{n,+}^{2^{*}}dx.
$$
By using Fatou's Lemma for $L\rightarrow+\infty$, we obtain
$$\left(\int_{B_{3R}(x)}v_{n,+}^{\frac{{2^{*}}^{2}}{2}}dx\right)^{\frac{2}{2^{*}}}\leq C\beta_{0}^{2}\int v_{n,+}^{2^{*}}dx$$
for all $n\in\N$ and for all $x\in\R^{N}$. This proves Claim \ref{afirmacaoregularidade}.

In what follows, we fix $R>0$ as in Claim \ref{afirmacaoregularidade}, $r_{m}:=\frac{2R}{2^{m}}$, 
$$
t:=\frac{{2^{*}}^{2}}{2(2^{*}-2)} \quad \mbox{and} \quad  \chi:=\frac{2^{*}(t-1)}{2t}>1.
$$

\begin{claim}
Consider $\beta>1$ arbitrary such that $v_{n,+}\in L^{\beta\frac{2^{*}}{\chi}}(B_{R+r_{m}}(x))$ for all $n\in\N$ and for some $m\in\N$. Then
\begin{equation}\label{desigualdade7}
|v_{n,+}|_{L^{2^{*}\beta}(B_{R+r_{m+1}}(x))}\leq C^{1/\beta}\beta^{1/2\beta}(1+4^{m})^{1/2\beta}|v_{n,+}|_{L^{2^{*}\frac{\beta}{\chi}}(B_{R+r_{m}}(x))}
\end{equation}
where $C>0$ is independently of $n,m\in\N$, $\beta>1$ and $x\in\R^{N}$.
\end{claim}
In fact, since $2^{*}\frac{\beta}{\chi}=\beta\frac{2t}{t-1}$, $v_{n,+}\in L^{\frac{2\beta t}{t-1}}(B_{R+r_{m}}(x))$ for all $n\in\N$. Consider $\eta_{x,m}\in C^{\infty}_{c}(\R^{N},[0,1])$ such that $\eta_{x,m}\equiv1$ in $B_{R+r_{m+1}}(x)$, $\eta_{x,m}\equiv0$ in $\R^{N}\setminus B_{R+r_{m}}(x)$ and $|\eta_{x,m}|_{\infty}<\frac{2}{r_{m+1}}$. Using $\eta=\eta_{x,m}$ in (\ref{desigualdade6}),  
$$|w_{L,n}|_{2^{*}}^{2}\leq C\beta^{2}\left(\int_{B_{R+r_{m}}(x)}|v_{n,+}|^{2\beta}dx+\int_{B_{R+r_{m}}(x)}v_{n,+}^{2^{*}-2}v_{n,+}^{2\beta}dx+\right.$$
$$\left.+\left(\frac{2}{r_{m+1}}\right)^{2}\int_{B_{R+r_{m}}(x)}v_{n,+}^{2\beta}dx\right)\leq C\beta^{2}\left((1+4^{m})\int_{B_{R+r_{m}}(x)}v_{n,+}^{2\beta}dx+\right.$$
$$\left.+\int_{B_{R+r_{m}}(x)}v_{n,+}^{2^{*}-2}v_{n,+}^{2\beta}dx\right)\leq C\beta^{2}\left((1+4^{m})\left(\int_{B_{3R}(0)}1dx\right)^{1/t}.\right.$$
$$.\left.\left(\int_{B_{R+r_{m}}(x)}v_{n,+}^{2\beta t/(t-1)}dx\right)^{(t-1)/t}+\left(\int_{B_{3R}(x)}v_{n,+}^{(2^{*}-2)t}dx\right)^{1/t}.\right.$$
$$\left.\left(\int_{B_{R+r_{m}}(x)}v_{n,+}^{2\beta t/(t-1)}dx\right)^{(t-1)/t}\right)\leq$$
$$\leq C\beta^{2}\left((1+4^{m})\left(\int_{B_{R+r_{m}}(x)}v_{n,+}^{2\beta t/(t-1)}dx\right)^{(t-1)/t}\right).$$
Thus
$$
|w_{L,n}|^{2}_{2^{*}}\leq C\beta^{2}(1+4^{m})|v_{n,+}|_{L^{2\beta t/(t-1)}(B_{R+r_{m}}(x))}^{2\beta}.
$$
Applying Fatou's Lemma as $L\rightarrow+\infty$ we get (\ref{desigualdade7}).
Consequently, by induction, 
\begin{equation}
|v_{n,+}|_{L^{2^{*}\chi^{m}}(B_{R+r_{m+1}}(x))}\leq C^{\sum_{i=1}^{m}\frac{1}{\chi^{i}}}\chi^{\sum_{i=1}^{m}\frac{i}{2\chi^{i}}}\prod_{i=1}^{m}(1+4^{i})^{\frac{1}{2\chi^{i}}}|v_{n,+}|_{L^{2^{*}}(B_{2R}(x))}
\end{equation}
Since $\left(\sum_{i=1}^{m}\frac{1}{\chi^{i}}\right)_{m}$ and $\left(\sum_{i=1}^{m}\frac{i}{\chi^{i}}\right)_{m}$ are convergent because $\chi>1$, and that 
$$
\prod_{i=1}^{m}(1+4^{i})^{\frac{1}{2\chi^{i}}}=4^{\sum_{i=1}^{m}\frac{log_{4}(1+4^{i})}{2\chi^{i}}}\leq4^{\sum_{i=1}^{m}\frac{log_{4}(4^{i+1})}{2\chi^{i}}}=4^{\sum_{i=1}^{m}\frac{i+1}{2\chi^{i}}},
$$
there exists $C>0$ independently of $n,m\in\N$ and $x\in\R^{N}$ such that
$$
|v_{n,+}|_{L^{2^{*}\chi^{m}}(B_{R}(x))}\leq C|v_{n,+}|_{L^{2^{*}}(B_{2R}(x))}.
$$
Now (\ref{regularidade}) follows by taking the limit of $m \to +\infty$.  
\end{proof}

\vspace{0.5 cm}

\begin{corollary}\label{corolariolimitacaozn}
For each  $\delta>0$ there exist $R>0$ such that $|v_{n}(x)|\leq \delta$ for all $x\in\R^{N}\setminus B_{R}(0)$ and $n \in \mathbb{N}$.
\end{corollary}
\begin{proof} By Lemma \ref{Lema2.13}, 
$$
|v_{n}|_{L^{\infty}(B_{R}(x))}\leq C|v_{n}|_{L^{2^{*}}(B_{2R}(x))}, \quad \mbox{for all} \quad n\in\N \quad \mbox{and} \quad x\in\R^{N}.	
$$
This fact combined with the limit $v_n \to v$ in $L^{2^{*}}(\R^{N})$ proves the result.
\end{proof}

\vspace{0.5 cm}

\noindent\textbf{Concentration of the solutions:}
\\

As $v\neq0$, we must have $|v_{n}|_{L^{\infty}(\R^{N})}\not\rightarrow0$. Hence, we can assume that $|v_{n} |_{L^{\infty}(\R^{N})}>\delta$ for any $\delta>0$ and $n \in \mathbb{N}$. In what follows, we fix $z_{n}\in\R^{N}$ verifying 
$$
|v_{n}(z_{n})|=\max_{x\in\R^{N}}|v_{n}(x)|.
$$
Since $v_{n}(x)=u_{n}(x+y_{n})$, the point $x_{n}:=z_{n}+y_{n}$ satisfies
$$
|u_{n}(x_{n})|=\max_{x\in\R^{N}}|u_{n}(x)|.
$$
From Corollary \ref{corolariolimitacaozn}, $(z_{n})$ is bounded in $\R^N$, then   
$$
\epsilon_{n}x_{n}=\epsilon_{n}z_{n}+\epsilon_{n}y_{n}\rightarrow z\in\mathcal{A}
$$
and 
$$
\lim_{n}A(\epsilon_{n}x_{n})=A(z)=A(0).
$$
\section{The case $N=2$.}

In this section we will consider the case where $f$ has an exponential critical growth. For this type of function, it is well  known that Trundiger-Moser type inequalities are key points to apply variational methods. In the present paper we will use a Trudinger-Moser type inequality for whole $\mathbb{R}^{2}$ due to Cao \cite{Cao} ( see also Ruf \cite{Ruf} ).

\begin{lemma} \noindent {\bf (Trudinger-Moser inequality for unbounded domains)}\label{Cao}
	For all $ u \in H^{1}(\mathbb{R}^{2})$, we have
	$$
	\int 
	\left(e^{\alpha\left|u\right|^{2}}-1 \right)dx
	< \infty,\,\,\,\, \mbox{ for every }\,\,\alpha >0.
	$$
	Moreover, if $\left| \nabla
	u\right|^{2}_{2}\leq 1,\,\left|u\right|_{2} \leq M < \infty $ and
	$\alpha < 4 \pi$, then there
	exists a positive constant $C=C(M,\alpha)$ such that
	$$
	\int \left(e^{\alpha\left|u\right|^{2}}-1\right)dx \leq C.
	$$
\end{lemma}

The reader can find other Trundiger-Moser type inequalities in \cite{CT},  \cite{masmoudi}, \cite{ish}, \cite{MS} and references therein

\vspace{0.5 cm}

As in the previous section, firstly we need to study the autonomous case.

\subsection{A result involving the autonomous problem.}
We consider the problem
$$
\left\{\begin{array}{l}
-\Delta u+V(x)u=\lambda f(u), \quad  x \in \R^{2}, \\ 
u\in H^{1}(\R^{2}),
\end{array}\right.\eqno{(AP)_{\lambda}^{exp}}
$$
where $f:\R\rightarrow\R$ satisfies $(f_1)-(f_5)$. Associated with this problem, we have the energy function $J_{\lambda}:H^{1}(\R^{2})\rightarrow\R$ given by
$$
J_{\lambda}(u)=\frac{1}{2}||u^{+}||^{2}-\frac{1}{2}||u^{-}||^{2}-\lambda\int F(u)dx.
$$
It is well known that $J_{\lambda}\in C^{1}(H^{1}(\R^{2}),\R)$ with 
$$
J_{\lambda}'(u)v=B(u,v)-\lambda\int f(u)vdx, \quad \forall u,v \in H^{1}(\R^2). 
$$
In the sequel, 
$$
\mathcal{N}_{\lambda}=\{u\in H^{1}(\R^{2})\setminus E^{-}\ ;\ J_{\lambda}'(u)u=J_{\lambda}'(u)v=0, \forall\ v\in E^{-}\}
$$
and
$$
d_{\lambda}=\inf_{\mathcal{N}_{\lambda}}J_{\lambda}.
$$
In \cite{AG}, Alves and Germano have proved that there exists a constant $\tau_0>0$ such that $(AP)_{\lambda}^{exp}$ has a ground state solution if 
\begin{equation}\label{condicaoF}
\lambda \geq A(0) \quad \mbox{and} \quad \tau \geq \tau_0,
\end{equation}
where $\tau$ was fixed in $(f_5)$. More precisely, it has been shown  that for $\lambda\geq A(0)$ and $\tau \geq \tau_0$, there exists $u_{\lambda}\in H^{1}(\R^{2})$ verifying
$$
J_{\lambda}'(u_{\lambda})=0 \text{ and }J_{\lambda}(u_{\lambda})=d_{\lambda}
$$
with
\begin{equation}\label{limitacaod2}
d_{\lambda}<\frac{\widetilde{A}^{2}}{2}
\end{equation}
where $\widetilde{A}<1/a$ and $a$ was given in (\ref{equivalente}). This restriction on $\tau$ has been mentioned  in Theorem \ref{T1}, and it will be assume in whole this section.

Moreover, the authors have proved that for all $u\in H^{1}(\R^{2})\setminus E^{-}$ the set $\mathcal{N}_{\lambda}\cap\widehat{E}(u)$ is a singleton set and the element of this set is the unique global maximum of $J_{\lambda}|_{\widehat{E}(u)}$, which means precisely that there exist uniquely $t^{*}\geq0$ and $v^{*}\in E^{-}$ such that
$$
J_{\lambda}(t^{*}u+v^{*})=\max_{w\in\widehat{E}(u)}J_{\lambda}(w)\hspace{0.5cm}\text{ and }\hspace{0.5cm}\{t^{*}u+v^{*}\}=\mathcal{N}_{\lambda}\cap\widehat{E}(u)
$$

As in the case $N\geq3$, we begin by studying the behavior of the function $\lambda\mapsto d_{\lambda}$.  

\begin{proposition}\label{continuidadec2}
The function $\lambda\mapsto d_{\lambda}$ is decreasing and continuous on $[A_0,+\infty)$.
\end{proposition}
\begin{proof}
The monotonicity of $\lambda\mapsto d_{\lambda}$ and some details of the proof are analogous to Proposition \ref{continuidadec} and \cite[Proposition 2.3]{AG2}. In order to get the limit $\lim_{n} d_{\lambda_n}=d_\lambda$, it suffices to consider $\lambda_{1}\geq\lambda_{2}\geq...\geq\lambda_{n}\rightarrow\lambda$. Let $u_{n}$ be a ground state solution of the problem $(AP)_{\lambda_n}^{exp}$. Let $t_{n}\geq0$ and $v_{n}\in E^{-}$ such that $t_{n}u_{n}+v_{n}\in\mathcal{N}_{\lambda}$. Consequently
$$
J_{\lambda}(t_{n}u_{n}+v_{n})=\max_{\widehat{E}(u_{n})}J_{\lambda}\geq d_{\lambda},
$$
and the same ideas explored in Proposition \ref{continuidadec} remain valid to show that $\left(\int f(u_{n})u_{n}dx\right)$ is bounded in $\R$. Now, arguing as in \cite[Lemma 3.11]{AG}, we see that  $(u_{n})$ is bounded in $H^{1}(\R^{2})$.

Note that there exist $(y_{n})$ in $\Z^{2}$, $r,\eta>0$ such that
\begin{equation}\label{lionsn2}
\int_{B_{r}(y_{n})}|u_{n}^{+}|^{2}dx>\eta, \quad \forall n \in \mathbb{N}.
\end{equation}
Otherwise, $u_{n}^{+}\rightarrow0$ in $L^{p}(\R^{2})$ for all $p>2$. Defining $w_{n}(x):=\widetilde{A}\frac{u_{n}^{+}(x)}{||u_{n}||}$ where $\widetilde{A}$ was given in (\ref{limitacaod2}), we have 
$$
||w_{n}||_{H^{1}(\R^{2})}\leq\widetilde{A}a<1, \quad \forall n \in \mathbb{N}.
$$
This fact permits to repeat the same approach found in \cite[Proposition 2.3]{AOM} to get  the limit
$$
\int F(w_{n})dx\rightarrow 0.
$$
As $w_{n}\in\widehat{E}(u_{n})$ and $u_{n}\in\mathcal{N}_{\lambda_{n}}$, it follows that 
$$
d_{\lambda}\geq d_{\lambda_{n}}=J_{\lambda_{n}}(u_{n})\geq J_{\lambda_{n}}(w_{n})=\frac{\widetilde{A}}{2}-\lambda_{n}\int F(w_{n})dx.
$$
Passing to the limit as $n\rightarrow+\infty$ we obtain $d_{\lambda}\geq\widetilde{A}/2$, which contradicts (\ref{limitacaod2}), and (\ref{lionsn2}) holds. If $\widetilde{u}_{n}(x):=u_{n}(x+y_{n})$, then $\widetilde{u}_{n}^{+}(x):=u_{n}^{+}(x+y_{n})$,  and by (\ref{lionsn2}), $\widetilde{u}_{n}^{+}\rightharpoonup u\neq0$. This implies that ${\cal V}:=\{\widetilde{u}_{n}^{+}\}_{n\in\N}$ satisfies $0\notin\overline{{\cal V}}^{\sigma(H^{1}(\R^{2}),H^{1}(\R^{2})')}$ and ${\cal V}$ is bounded in $H^{1}(\R^{2})$. We proceed as in Proposition \ref{continuidadec} to conclude $(t_{n}u_{n}+v_{n})$ is bounded and $d_{\lambda_{n}}\leq d_{\lambda}+o_{n}$. This finishes the proof. 
\end{proof}

\subsection{Existence of ground state for problem $(P)_\epsilon$.}

The three first results this section follow as in the case $N \geq 3$, then we will omit their proofs.

\begin{lemma}
The limit $\lim_{\epsilon\rightarrow0}c_{\epsilon}=c_{0}$ holds. Moreover, if $w_{0}$ is a ground state solution of the problem $(P)_{0}$ and let $t_{\epsilon}\geq0$ and $v_{\epsilon}\in E^{-}$ such that $t_{\epsilon}w_{0}+v_{\epsilon}\in\M_{\epsilon}$. Then
	$$t_{\epsilon}\rightarrow1\ \text{ and }\ v_{\epsilon}\rightarrow 0$$
	as $\epsilon\rightarrow0$. 
\end{lemma}

\begin{corollary}
	There exists $\epsilon_{0}>0$ such that 
	$$
	c_{\epsilon}<d_{A_{\infty}}\ \text{ and }\ c_{\epsilon}<\frac{\widetilde{A}^{2}}{2}, \quad \mbox{for all} \quad \epsilon \in (0,\epsilon_{0}).
	$$
\end{corollary}

\begin{proposition}
	There exists a bounded sequence $(u_{n}) \subset \M_{\epsilon}$ such that $(u_{n})$ is $(PS)_{c_{\epsilon}}$ for $I_{\epsilon}$.
\end{proposition}

Now we are ready to prove the existence of solution for $\epsilon$ small enough.

\begin{theorem}
Problem  $(P)_{\epsilon}$ has a ground state solution for $\epsilon\in(0,\epsilon_{0})$ .
\end{theorem}

\begin{proof}
To begin with, we claim that there are $(z_{n}) \subset \Z^{2}$ and $r,\eta>0$ such that
\begin{equation}\label{desigualdadelions2}
\int_{B_{r}(z_{n})}|u_{n}^{+}|^{2}dx>\eta, \quad \forall n \in \mathbb{N}.
\end{equation} 
In fact, if the claim does not hold, we must have $u_{n}^{+}\rightarrow0$ in $L^{p}(\R^{2})$ for all $p\in(2,+\infty)$. Since $u_{n}\in\M_{\epsilon}$, by (\ref{umaislimitadainferiormente}), $||u_{n}^{+}||^{2}\geq2c_{\epsilon}\geq2c_{0}$. Setting $\widetilde{w}_{n}(x):=\widetilde{A}\frac{u_{n}^{+}}{||u_{n}^{+}||}$ and arguing as in Proposition \ref{continuidadec2}, we find $c_{\epsilon}\geq\frac{\widetilde{A}^{2}}{2}$, which is a contradiction. Therefore (\ref{desigualdadelions2}) holds.

\begin{claim} \label{Z*}
$(z_{n})$ is bounded in $\R^{2}$.
\end{claim}
Suppose $|z_{n}|\rightarrow+\infty$ and define $w_{n}(x):=u_{n}(x+z_{n})$. From (\ref{desigualdadelions2}), we can suppose that $w_{n}\rightharpoonup w\neq0$ in $H^{1}(\R^2)$. As it was done in (\ref{limitacaofunun}),  $\left(\int f(w_{n})w_{n}dx\right)$ is bounded in $L^{1}(\R^2)$. By \cite[Lemma 2.1]{djairo}, 
$$
f(w_{n})\rightarrow f(w)\ \ \text{ in }L^{1}(B),
$$
for all $B\subset\R^{2}$ bounded Borel set. Now, we repeat the same idea explored in Claim \ref{znlimitada} to deduce that $w$ is a critical point of $J_{A_{\infty}}$ with $d_{A_{\infty}}\leq c_{\epsilon}$, which is absurd. This proves the Claim \ref{Z*}.

To conclude the proof we proceed as in Theorem \ref{existenciapontocritico} to prove that the weak limit of $(u_n)$ is a ground state solution for $I_{\epsilon}$.
\end{proof}

\subsection{Concentration of the solutions.}
In this section we fix  $\epsilon_{n}\rightarrow 0$ with $\epsilon_{n}\in(0,\epsilon_{0})$ for all $n\in\N$. By results of the last section, for each $n\in\N$ there exists $u_{n}$ in $H^{1}(\R^{2})$ such that
$$
I_{n}(u_{n})=c_{n}\ \ \text{ and }\ \ I_{n}'(u_{n})=0,
$$
with the notation $I_{n}:=I_{\epsilon_{n}}$ and $c_{n}:=c_{\epsilon_{n}}$.

\begin{lemma}
The sequence $(u_{n})$ is bounded in $H^{1}(\R^{2})$.
\end{lemma}
\begin{proof}
See proof of \cite[Lemma 3.11]{AG}.
\end{proof}

\begin{lemma}
There are $r,\eta>0$ and $(y_{n}) \subset \Z^{2}$ such that
\begin{equation}\label{desigualdadelions3}
\int_{B_{r}(y_{n})}|u_{n}^{+}|^{2}dx>\eta.
\end{equation}
\end{lemma}
\begin{proof}
See proof of (\ref{desigualdadelions2}). 
\end{proof}

\vspace{0.5 cm}

From now on, we set $v_{n}(x):=u_{n}(x+y_{n})$. Then,  by (\ref{desigualdadelions3}),  $v_{n}\rightharpoonup v\neq 0$ in $H^{1}(\R^2)$ for some subsequence.

\begin{lemma}
The sequence $(\epsilon_{n}y_{n})$ is bounded in $\R^{2}$. Moreover, $I'_{0}(v)=0$ and if $\epsilon_{n}y_{n}\rightarrow z\in\R^{2}$ then $z\in\mathcal{A}$ or equivalently $A(z)=A(0)$.
\end{lemma}
\begin{proof}
As in the previous section, $(f(u_{n})u_{n})$ is bounded in $L^{1}(\R^{2})$. Then, by \cite[Lemma 2.1]{djairo}, 
$$
f(u_{n})\rightarrow f(u)\ \text{ in }\ L^{1}(B),
$$
for all bounded Borel set $B\subset\R^{2}$ . The above limit permits to repeat the same arguments explored in Lemma \ref{enynlimitado}.
\end{proof}

\vspace{0.5 cm}

Our next proposition follows with the same idea explored in Proposition \ref{funcaodominada}, then we omit its proof.
\begin{proposition}
There exists $h\in L^{1}(\R^{2})$ and a subsequence of $(v_n)$ such that 
$$
|f(v_{n}(x))v_{n}(x)|\leq h(x), \quad \mbox{for all} \quad x\in\R^{2} \quad \mbox{and} \quad n\in\N.
$$
\end{proposition}

As an immediate consequence of the last lemma, we have the following corollary
\begin{corollary}
$v_{n}\rightarrow v$ in $L^{q}(\R^{2})$ where $q$ was given in $(f_5)$.
\end{corollary}
\begin{proof}
It suffices to note that $f(v_{n})v_{n}\geq\theta F(v_{n}) \geq \theta\tau|v_{n}|^{q}$ for all $n \in \mathbb{N}$ and $v_n(x) \to v(x)$ a.e in $\mathbb{R}^N$.
\end{proof}

\vspace{0.5 cm}

The next lemma have been motivated by an inequality found  \cite[Lemma 2.11]{DORUF}, however it is a little different, because we need to adapt it for our problem.

\begin{lemma}\label{desigualdadechave}
For all $t,s\geq0$ and $\beta\in(0,1]$, 
$$
ts\leq\left\{\begin{array}{ll}4(e^{t^{2}}-1)(ln^{+}s)+s(ln^{+}s)^{1/2}, & \text{ if }s>e^{1/4} \\ e^{1/4}ts^{\beta}, & \text{ if }s\in[0,e^{1/4}]. \end{array}\right.
$$
\end{lemma}
\begin{proof}
From \cite[Lemma 2.11]{DORUF}, if $s>e^{1/4}$ then $ln^{+}s>1/4$ and
$$
ts\leq(e^{t^{2}}-1)+s(ln^{+}s)^{1/2}\leq4(e^{t^{2}}-1)(ln^{+}s)+s(ln^{+}s)^{1/2}.
$$
For $s\in[0,1)$, we have $ts\leq ts^{\beta}\leq e^{1/4}ts^{\beta}$, and if $s\in[1,e^{1/4}]$, then $ts\leq te^{1/4}\leq e^{1/4}ts^{\beta}$. This proves the inequality.
\end{proof}

\begin{proposition}\label{limiteforte2}
$v_{n}\rightarrow v$ in $H^{1}(\R^{2})$.
\end{proposition}
\begin{proof}
To begin with, by $(f_1)$, there exists $K>0$ such that
$$
|f(t)|\leq\Gamma e^{1/4} \Longrightarrow |f(t)|^{2}\leq K f(t)t.
$$
On the other hand, 
$$
\left(|f(v_{n})|\chi_{[0,e^{1/4}]}\left(\frac{1}{\Gamma}|f(v_{n})|\right)\right)^{2}=|f(v_{n})|^{2}\chi_{[0,\Gamma e^{1/4}]}(|f(v_{n})|)\leq
$$
$$
\leq Kf(v_{n})v_{n}\leq Kh\in L^{1}(\R^{2}).
$$
Thus, there exists $\widetilde{h}\in L^{2}(\R^{2})$ such that
$$
|f(v_{n})|\chi_{[0,e^{1/4}]}\left(\frac{1}{\Gamma}|f(v_{n})|\right)\leq\widetilde{h}, \quad \forall n \in \mathbb{N}.
$$
In what follows, fixing $\alpha >0$ such that $\frac{\alpha^{2} q}{q-1} \sup_{n \in \N}\|v_n^+\|^{2}_{H^{1}(\R^2)}<1$, the Lemma \ref{Cao} guarantees that 
$$
b_{n}:=(e^{\alpha^{2}|v_{n}^{+}|^{2}}-1)\in L^{\frac{q}{q-1}}(\R^{2}) \quad \mbox{and} \quad |b_n|_{\frac{q}{q-1}} \leq C
$$
for all $n \in \mathbb{N}$ and some $C>0$. Applying the Lemma \ref{desigualdadechave} for $t=\alpha|v_{n}^{+}|$, $s=\frac{1}{\Gamma}|f(v_{n})|$ and $\beta=1$, we obtain
$$
|f(v_{n})v_{n}^{+}|=\frac{\Gamma}{\alpha}\frac{|f(v_{n})|}{\Gamma}\alpha|v_{n}^{+}|\leq\frac{\Gamma}{\alpha}4(e^{\alpha^{2}|v_{n}^{+}|^{2}}-1)\left(ln^{+}\left(\frac{1}{\Gamma}|f(v_{n})|\right)\right)+
$$
$$+\frac{1}{\alpha}|f(v_{n})|\left(ln^{+}\left(\frac{1}{\Gamma}|f(v_{n})|\right)\right)^{1/2}+e^{1/4}|v_{n}^{+}||f(v_{n})|\chi_{[0,e^{1/4}]}\left(\frac{1}{\Gamma}f(v_{n})\right)\leq$$
$$
\leq\frac{16\Gamma\pi}{\alpha}b_{n}|v_{n}|^{2}+\frac{\sqrt{4\pi}}{\alpha}f(v_{n})v_{n}+e^{1/4}|v_{n}^{+}|\widetilde{h}.
$$
Since $b_{n}\rightharpoonup b$ in $L^{\frac{q}{q-1}}(\R^{2})$ and $v_{n}\rightarrow v$ in $L^{q}(\R^{2})$, we have that $(b_{n}|v_{n}|^{2})$ is strongly convergent in $L^{1}(\R^{2})$. Here, we have used the fact that $b_n|v_n|^{2} \geq 0$ and $v_n(x) \to v(x)$ a.e in $\mathbb{R}^N$. Analogously $(|v_{n}^{+}|\widetilde{h})$ converges in $L^{1}(\R^{2})$. Consequently there is $H_1 \in L^{1}(\R^{2})$ such that, for some subsequence,
$$
|f(v_{n})v_{n}^{+}|\leq H, \quad \forall n \in \mathbb{N}.
$$
The same argument works to show that there exists $H_2 \in L^{1}(\R^{2})$ such that, for some subsequence,
$$
|f(v_{n})v_{n}^{-}|\leq H_2, \quad \forall n \in \mathbb{N}.
$$
As an consequence of the above information,
$$
f(v_{n})v_{n}^{+} \to f(v)v^{+} \quad \mbox{and} \quad f(v_{n})v_{n}^{-} \to f(v_{n})v^{-} \quad \mbox{in} \quad L^{1}(\R^2).
$$
Now, recalling that $I_0'(v)=I_n'(v_n)v_n^{+}=I_n'(v_n)v_n^{-}=0,  v_n^{+}\rightharpoonup v^{+}$, and $v_n^{-}\rightharpoonup v^{-}$ in $H^{1}(\mathbb{R}^2)$, we get the desired result. 
\end{proof}

\begin{lemma} \label{controleuniforme}
	For all $n\in\N$, $v_{n}\in C(\R^{2})$. Moreover, there exist $G\in L^{3}(\R^{2})$, $C>0$ independently of $x\in\R^{2}$ and $n\in\N$ such that
	$$
	||v_{n}||_{C(\overline{B_{1}(x)})}\leq C|G|_{L^{3}(B_{2}(x))}, \quad \mbox{for all} \quad n\in\N \quad \mbox{and} \quad x\in\R^{2}.
	$$
	Hence, there exists $C>0$ such that $|v_{n}|_{L^{\infty}(\R^{2})}\leq C$ and
	$$
	|v_{n}(x)|\rightarrow 0\text{ as }|x|\rightarrow+\infty, \quad \mbox{uniformly in} \quad n \in \mathbb{N}.
	$$

\end{lemma}
\begin{proof}
We know that there are $C_{1},C_{2}>0$ such that
$$
|f(t)|\leq C_{1}|t|+C_{2}(e^{5\pi t^{2}}-1) \quad \forall t \in \mathbb{R}.
$$
By Proposition \ref{limiteforte2}, there exists $H\in H^{1}(\R^{2})$ such that $|v_{n}(x)|\leq H(x)$ for all $n\in\N$ and $x\in\R^{2}$. Setting
$$
G:=(||V||_{\infty}+A(0)C_{1})H+A(0)C_{2}(e^{5\pi H^{2}}-1)\in L^{3}(\R^{2})
$$
it follows that 
$$
|A(\epsilon_{n}x+\epsilon_{n}y_{n})f(v_{n})-V(x)v_{n}|\leq G(x), \quad \mbox{for all} \quad n \in \mathbb{N} \quad \mbox{and} \quad x \in \R^{2}.
$$
Since
$$
\left\{\begin{array}{l}
-\Delta v_{n}+V(x)v_{n}=A(\epsilon_{n}x+\epsilon_{n}y_{n})f(v_{n}),\quad \mbox{in} \quad \R^{2}, \\
v_{n}\in H^{1}(\R^{2})
\end{array}\right.
$$
From \cite[Theorems 9.11 and 9.13]{GT}, there exists $C_{3}>0$ independently of $x\in\R^{2}$ and $n\in\N$ such that $v_{n}\in W^{2,3}(B_{2}(x))$ and
\begin{equation}\label{regularidade1}
||v_{n}||_{W^{2,3}(B_{2}(x))}\leq C_{3}|G|_{L^{3}(B_{2}(x))},\ \ \ \text{for all }n\in\N.
\end{equation}
On the other hand, from continuous embedding $W^{2,3}(B_{2}(x)) \hookrightarrow C(\overline{B_{1}(x)})$, there is $C_{4}>0$ independently of $x\in\R^{2}$ such that  
\begin{equation}\label{regularidade2}
||u||_{C(\overline{B_{1}(x)})}\leq C_{4}||u||_{W^{2,3}(B_{2}(x))},\ \ \ \text{for all }u\in W^{2,3}(B_{2}(x)).
\end{equation}
The result follows from (\ref{regularidade1}) and (\ref{regularidade2}).
\end{proof}

\vspace{0.5 cm}

\noindent\textbf{Concentration of the solutions:}
\\ 

The proof of the concentration follows with the same idea explored in the case $N\geq3$, then we omit its proof.

\end{document}